\documentclass[11pt]{amsart}
\usepackage{latexsym}
\usepackage{aeguill}
\usepackage[latin1]{inputenc}
\usepackage[T1]{fontenc}
\usepackage{amsfonts}
\usepackage{amssymb}
\usepackage{xspace}
\usepackage{graphicx}
\usepackage{amsmath}
\usepackage{amsthm}
\usepackage{endnotes}
\usepackage{tikz}
\usepackage{pgf}
\usepackage{yfonts}

\usepackage[all]{xy}
\usepackage{amscd}
\usetikzlibrary{matrix}

\newcommand{\cal}{\mathcal}

\renewcommand\>{\rangle}

\def\epsilon{\varepsilon}
\def\phi{\varphi}

\def\hat{\widehat}

\newcommand{\Out}{\mbox{Out}}
\newcommand{\Aut}{\mbox{Aut}}

\newcommand{\FN}{F_n}   

\newcommand{\Z}{\mathbb Z}

\def\strutdepth{\dp\strutbox}
\def \ss{\strut\vadjust{\kern-\strutdepth \sss}}
\def \sss{\vtop to \strutdepth{
\baselineskip\strutdepth\vss\llap{$\diamondsuit\;\;$}\null}}

\def\strutdepth{\dp\strutbox}
\def \sst{\strut\vadjust{\kern-\strutdepth \ssss}}
\def \ssss{\vtop to \strutdepth{
\baselineskip\strutdepth\vss\llap{$\spadesuit\;\;$}\null}}

\def\strutdepth{\dp\strutbox}
\def \ssh{\strut\vadjust{\kern-\strutdepth \sssh}}
\def \sssh{\vtop to \strutdepth{
\baselineskip\strutdepth\vss\llap{$\heartsuit\;\;$}\null}}

\def\qed{\hfill\rlap{$\sqcup$}$\sqcap$\par}
\def\bar{\overline}

\def\strutdepth{\dp\strutbox}
\def \ss{\strut\vadjust{\kern-\strutdepth \sss}}
\def \sss{\vtop to \strutdepth{
\baselineskip\strutdepth\vss\llap{$\diamondsuit\;\;$}\null}}

\def\strutdepth{\dp\strutbox}
\def \sst{\strut\vadjust{\kern-\strutdepth \ssss}}
\def \ssss{\vtop to \strutdepth{
\baselineskip\strutdepth\vss\llap{$\spadesuit\;\;$}\null}}

\def\qed{\hfill\rlap{$\sqcup$}$\sqcap$\par}

\newtheorem*{thm*}{Theorem}

\newtheorem{thm}{Theorem}[section]
\newtheorem{cor}[thm]{Corollary}
\newtheorem{lem}[thm]{Lemma}
\newtheorem{prop}[thm]{Proposition}
\newtheorem{convention}[thm]{Convention}

\theoremstyle{definition}
\newtheorem{defn}[thm]{Definition}

\newtheorem{rem}[thm]{Remark}

\theoremstyle{remark}

\numberwithin{equation}{section}

\begin{document}

\author[K.~Ye]{Kaidi Ye}

\title[Which polynomially growing $\hat \phi$ is geometric ?]
{When is a polynomially growing automorphism of $\FN$ geometric ?}

\begin{abstract} 
The main result of this paper is an algorithmic answer to the question raised in the title, up to replacing the given $\hat \phi \in \Out(\FN)$ by a positive power.

In order to provide this algorithm, it is shown that every polynomially growing automorphism $\hat \phi$ can be represented by an 
iterated Dehn twist on some graph-of-groups $\cal G$ with $\pi_1 \cal G = \FN$. One then uses results of two previous papers \cite{KY01, KY02} as well as some classical results such as the Whitehead algorithm to prove the claim.

\end{abstract}


\maketitle

\tableofcontents












\section{Introduction}

Let $\phi$ be an automorphism of a free group $F_n$ of finite rank $n \geq 2$. One says that $\phi$ (or the associated outer automorphism $\hat \phi \in \Out(\FN)$) is {\em geometric} if there is a positive answer to the following:

\medskip
\noindent
{\em Question: Does there exist a surface $S$ with $\pi_1 S \cong F_n$ and a homeomorphism $h: S \to S$ with $h_* = \hat\phi$ ?}

\medskip
 
The main purpose of this paper is to give an algorithmic answer to 
this question, for the case that $\phi$ has polynomial growth, and modulo replacing $\phi$ by some higher iterate.

\smallskip

To do this, we describe in this paper several ``subalgorithms'' which, when properly put together, fulfill this purpose. Some of these sub-algorithms have some interest in themselves, which we describe now:

In section \ref{sec-5} we define iteratively a class of automorphisms of $F_n$, called 
{\em iterated Dehn twist automorphisms (of some level $k \geq 1$)}, 
which are given by a graph-of-groups $\cal G$ with trivial edge groups, and an automorphism







$H: \cal G \to \cal G$ which acts trivially on the underlying graph and induces on each vertex group $G_v$ an 
iterated Dehn twist automorphisms of level $k_v \leq k -1$. For $k = 1$ the edge groups may be cyclic, and the vertex group automorphisms 
and edge group automorphisms
are the identity. Thus for $k = 1$ the resulting automorphism is a Dehn twist automorphism in the traditional sense.
We show (compare Proposition \ref{k-level-rep}):

\begin{prop}
\label{thm1}
Every polynomially growing automorphism $\hat \phi \in \Out(\FN)$ has a positive power 
$\hat{\phi}^t$
which is represented by an 
iterated Dehn twist automorphism $H$ of some level $k \geq 1$.  

All the data for $H$ (including for all lower levels through iteratively passing to the induced vertex group automorphisms) can be derived algorithmically from a relative train track representative of 
$\hat{\phi}$
as given by 
Bestvina-Feighn-Handel
\cite{BFH00}.

It can be derived from \cite{BFH00} that the above exponent $t = t_n \geq 1$ can be determined 
depending only on $n$ and not on the particular choice of $\hat \phi$.
\end{prop}


We then use our results from \cite{KY01}, \cite{KY02} to derive algorithmically from an 
iterated Dehn twist representative of level $k \geq 2$ either an 
iterated Dehn twist representative of 
strictly 
lower level, or else a conjugacy class $[w] \subset \FN$ that grows at least quadratically under iteration of $\hat \phi$. Thus we obtain 
(compare Theorem \ref{dichotomy0} and Corollary \ref{dichotomy0-cor}):



\begin{thm}
\label{thm2}
Every linearly growing automorphism of $\FN$ has a positive power 
which is 
a Dehn twist automorphism.

More precisely: 
From any iterated Dehn twist representative of some $\hat \phi \in \Out(\FN)$ one can derive algorithmically either the fact that $\hat \phi$ has at least quadratic growth, or else all the data of a graph-of-groups decomposition $\FN \cong \pi_1 \cal G$ as well as a Dehn twist $H: \cal G \to \cal G$ with 
$H_* = 
\hat \phi$.
\end{thm}

In the linearly growing case
one then uses work of Cohen-Lustig \cite{CL02} to derive algorithmically from $H$ and $\cal G$ an {\em efficient} Dehn twist representative of 
$\hat \phi$, and derive 
in section \ref{sec-6} 
from its uniqueness properties (see Theorem \ref{efficient-uniqueness}) the following:

\begin{prop}
\label{efficient-surface}
An efficient Dehn twist $H: \cal G \to \cal G$ represents a geometric automorphism of $\FN$ if and only if for every vertex $v$ of $\cal G$ the 
family of 
``edge generators'' $f_{e_i}(g_{e_i})$, where 
$e_i$ is any edge with terminal vertex $v$ and $g_e$ is a generator of $G_e \cong \Z$, 
is 
a boundary family in the vertex group $G_v$. 
\end{prop}

Here a {\em boundary family} 
is a family of elements in a free group $F_m$ which satisfy up to conjugation the sufficient and necessary condition that are well 
known
for elements which represent boundary components of an orientable or non-orientable surface (see 
Definition \ref{bdy-family}). 
This condition in turn can be decided algorithmically for any finite family of elements by the Whitehead algorithm on $G_v$, so that we obtain
(see section \ref{sec-6}):

\begin{cor}
\label{CLW}
For any Dehn twist $H: \cal G \to \cal G$ on a free group $\pi_1 \cal G =\FN$ it can be decided algorithmically whether 
$H_*$ 
is geometric or not.
\end{cor}

Putting all of the above together now gives:

\begin{thm}
\label{geometricYN}
There exists an algorithm which decides, for any given polynomially growing automorphism $\hat \phi \in \Out(\FN)$, whether $\hat \phi^{t_n}$ is geometric or not.
\end{thm}

The concrete terms of this algorithm are presented in detail in the last section of this paper.

\subsection*{Acknowledgements}
${}^{}$

This paper is part of my PhD thesis, 
and I would like to thank 
my advisors, Arnaud Hilion and Martin Lustig, 
for their advice and encouragement.


\section{
Preliminaries}

Throughout this paper $\FN$ will denote a free group of finite rank $n \geq 2$.

An outer automorphism $\hat{\phi} \in Out(F_n)$ is said to be {\em geometric} if it is induced by a homeomorphism on a surface. In other words, there exist a surface homeomorphism $h: S \rightarrow S$ and an isomorphism $\theta: \pi_{1}(S) \rightarrow F_n$ 
such that the following diagram commutes up to inner automorphisms,
\[
\begin{CD}
\pi_{1}(S) @> h_*>> \pi_{1}(S) \\
@V\theta VV @V \theta VV \\
F_n @>> \phi > F_n
\end{CD}
\]
where 
$\phi \in \Aut(\FN)$ is some representative of $\hat \phi$.

\subsection{Maps of Graphs and Topological Representatives}

${}^{}$



In this paper a {\em graph} $\Gamma$ is always assumed to be finite and connected, unless otherwise stated. We 
denote the set of its vertices by $V(\Gamma)$, and the set of its oriented edges by $E(\Gamma)$.
For any 
edge $e \in E(\Gamma)$ we denote by $\bar e \in E(\Gamma)$ the same edge with reverted orientation. Furthermore, $\tau(e)$ denotes the terminal vertex of $e$, and $\iota(e)$ its initial vertex.

A {\em path} (or {\em edge path}) in $\Gamma$ is either a single vertex, in which case the path is said to be {\em trivial}, or a non-empty sequence of edges $e_1 e_2 \dots e_k$ such that $\tau(e_i)=\iota(e_{i+1})$ for $1 \leq i \leq k-1$. 
A path is {\em reduced} if it does not contain a subpath of the form $e \bar{e}$ for some edge $e \in E(\Gamma)$.

A {\em graph map} $f: \Gamma \to \Gamma'$ is a map which sends vertices to vertices and edges to edge paths, which may or may not be reduced.

\begin{defn}
A {\em marked graph} refers to a pair $(\Gamma, \theta)$ where $\Gamma$ is a graph and the {\em marking} $\theta: \FN \overset{\cong}{\longrightarrow} \pi_1\Gamma$ is an isomorphism.
\end{defn}

\begin{defn}[Topological Representative]
Let $\hat{\phi} \in Out(F_n)$ be an outer automorphism of $\FN$. 
A {\em topological representative of $\hat{\phi}$ with respect to a marking 
$\theta: \FN \overset{\cong}{\longrightarrow} \pi_1\Gamma$}
is a homotopy equivalence $f: \Gamma \rightarrow \Gamma$ which determines the outer automorphism $\hat \phi$ on its fundamental group $\pi_{1}\Gamma$ (i.e. $f_* = \theta \,  \hat \phi \, \theta^{-1}$). Furthermore one requires 
that $f(e)$ is reduced and not a vertex, for every $e \in E(\Gamma)$.
\end{defn}

A {\em filtration} for a topological representative $f: \Gamma \rightarrow \Gamma$ is an increasing sequence of invariant subgraphs 
$\Gamma_0 \subset \Gamma_1 \subset \Gamma_2 \subset \dots \subset \Gamma_m=\Gamma$. Note that each $\Gamma_i$ is not necessarily connected. 
Each 
(possibly non-connected) 
subgraph $H_i= cl(\Gamma_i \smallsetminus \Gamma_{i-1})$ is referred to as the {\em i-th stratum}.
We usually assume that the edges of $\Gamma$ are labelled so that those in $H_i$ have smaller label than those in $H_{i+1}$.




\medskip

Bestvina-Handel \cite{BH92} (later improved by Bestvina-Feighn-Handel \cite{BFH00}) have shown that every automorphism $\hat \phi \in \Out(\FN)$ has a topological representative with very strong further properties. These {\em relative train track representatives} have strata $H_i$ which are either exponentially growing and 
have a ``train track'' property, or else they are polynomially growing
(relative to lower train track strata). 
Since in this paper we are only concerned with polynomially growing $\hat \phi$, we restrict ourselves here to quote a special case of their general result.

\begin{thm}[\cite{BFH00} Theorem~5.1.5]
\label{top-rep}
For every  polynomially-growing outer automorphism $\hat{\phi} \in Out(F_{n})$ one can determine algorithmically 
for a positive power 
$\hat \phi^t$ 
of $\hat{\phi}$ 
a topological representative $f: \Gamma \rightarrow \Gamma$ with a filtration $V(\Gamma) =\Gamma_{0} \subset \Gamma_{1} \subset \ldots \subset \Gamma_{m} = \Gamma$, which has the following properties:
\begin{enumerate}
\item the graph $\Gamma$ has no valence-one vertex;
\item all vertices $v \in V(\Gamma)$ are fixed by the map $f$;
\item every stratum $H_{i}$ 
consists of a single edge $e_i$ such that $f(e_{i})=e_{i}u_{i}$, where $u_{i} \subset \Gamma_{i-1}$ is a closed path.
\end{enumerate}
\end{thm}

It follows from going through the material in \cite{BFH00} that the exponent $t \geq 1$ in the above theorem can be chosen a priori, i.e. $t$ does not depend on $\hat \phi$ but only on the rank $n$ of the free group $\FN$, and that this exponent $t=t_n$ can be determined algorithmically for any $n \geq 2$.



\subsection{Graph-of-groups and Dehn twists}
\label{sec-gog's}

${}^{}$

In this subsection we will briefly recall some basic definitions of graph-of-groups and Dehn twists,
which will be used in later sections. We refer the readers to  \cite{CL02}, \cite{Serre}, \cite{KY01}, \cite{KY02} for more detailed information and discussions.
 
A {\em graph-of-groups} $\cal{G}=(\Gamma, (G_v)_{v \in V(\Gamma)}, (G_e)_{e \in E(\Gamma)}, (f_e)_{e\in E(\Gamma)})$ consists of a finite connected graph $\Gamma = \Gamma(\cal G)$, 
a {\em vertex group} $G_v$ for each vertex $v$ of $\cal G$ (by which we mean ``of $\Gamma(\cal G)''$), an {\em edge group} $G_e$ for each edge $e$ of $\cal G$, and a family of {\em edge monomorphisms} $f_e: G_e \to G_{\tau(e)}$.
%
%
For every $e \in E(\Gamma)$, we require $G_e=G_{\bar{e}}$. 

To $\cal G$ there is canonically associated a {\em path group} $\Pi(\cal G)$ which is the free product of all $G_v$ with the free group over $E(\Gamma)$, subject to the relations $t_{\bar e} = t_e^{-1}$ and $f_{\bar e}(g) = t_e f_e(g) t_e^{-1}$ for all $e \in E(\Gamma)$ and $g \in G_e$. For any $v \in V(\Gamma)$ there is a well defined {\em fundamental group} $\pi_1(\cal G, v) \subset \Pi(\cal G)$, and they are all naturally conjugate to each other in $\Pi(\cal G)$.


A {\em graph-of-groups automorphism} $H: \cal{G} \rightarrow \cal{G}$ is given by of a graph automorphism $H_{\Gamma}: \Gamma \rightarrow \Gamma$, a group isomorphism $H_v : G_{v} \rightarrow G_{H_{\Gamma}(v)} $ for each vertex $v$ of $\cal G$, a group isomorphism $H_e: G_e \rightarrow G_{H_{\Gamma}(e)}$ for each edge $e$ of $\cal G$,  and a correction term $\delta(e) \in G_{\tau(H_{\Gamma}(e))}$ for every edge $e$ of $\cal G$ which satisfies
 $$H_{\tau(e)}f_{e}=ad_{\delta(e)}f_{H_{\Gamma}(e)}H_{e}$$
where $ad_g: \FN \to \FN$ denotes conjugation with $g$. 

The isomorphism $H$ induces canonically isomorphisms $H_*: \Pi(\cal G) \to \Pi(\cal G)$ and $H_{*v}: \pi_1(\cal G,v) \to \pi_1(\cal G, v)$ as well as an outer automorphism $\hat H$ of $\pi_1 \cal G$ which is independent of the choice of the base point $v$. 
The latter is sometimes also denoted by $H_*$.

\medskip

A {\em Dehn twist} $D: \cal{G} \rightarrow \cal{G}$ on graph-of-groups $\cal G$
(defined in \cite{KY01} as ``general Dehn twist'')
is an automorphism of $\cal G$ where the graph automorphism $D_{\Gamma}$,
all vertex groups automorphisms $D_v$ and all edge groups automorphisms $D_e$ are the identity. Furthermore, one requires that for each edge $e$ the correction term $\delta(e)$ is contained in the centralizer of $f_e(G_e)$ in $G_{\tau(e)}$.

For free groups the last condition implies that for a non-trivial correction term 
$\delta(e)$
the edge group $G_e$ 
must be either trivial or infinite cyclic. Both cases occur when practically working with {\em Dehn twist automorphisms} of $\FN$, i.e. automorphisms $\hat \phi \in \Out(\FN)$ which satisfy for some identification $\FN = \pi_1 \cal G$ that $\hat \phi = \hat D$.

For every edge $e$ of $\cal G$ with $G_e \cong \Z$ one calls the element $z_e: = \gamma_{\bar e} \gamma_e^{-1}$ the {\em twistor} of $e$, where $\gamma_e \in G_e$ is defined by $f_e(\gamma_e) = 
\delta(e)$.
If $G_e \cong \Z$ for all edges $e$ of $\cal G$ then $D$ is called in \cite{KY01} a {\em classical} Dehn twist.
In this case we chose
%
for each edge $e$ of $\cal G$
an {\em edge generator} $g_{e}$ of the edge group $G_e$ and determine a {\em twist exponent} $n(e) \in \Z$ such that $z(e)=g_{e}^{n(e)}$. We use the convention that $g_e$ is always picked so that $n(e) \geq 0$. 


A {\em partial Dehn twist relative to a subset of vertices $\cal{V} \subset V(\Gamma(\cal G))$} is a graph-of-groups isomorphism $H: \cal G \to \cal G$ on a graph-of-groups $\cal G$ with trivial edge groups, which 
satisfies all conditions of a Dehn twist except that the vertex group automorphisms $H_v$ for any $v \in \cal{V}$ may not be the identity. We also require for such a partial Dehn twist that $\cal G$ is {\em minimal}, i.e. there is no proper subgraph-of-groups $\cal G'$ of $\cal G$ where the injection induces an isomorphism $\pi_1 \cal G' \cong \pi_1 \cal G$.\footnote{\,Unfortunately this minimality condition was omitted by mistake from Definition 3.7 of \cite{KY02}}


\subsection{Efficient Dehn twists}
\label{sec-efficient}

${}^{}$

In \cite{CL02} Cohen-Lustig introduced for free groups $\FN \cong \pi_1 \cal G$ a particular class of Dehn twists which have rather special and nice properties:

\begin{defn}[Efficient Dehn twist \cite{CL02}]
\label{efficient-D-twists-defn}
A classical Dehn twist $D=D(\cal{G},(z_{e})_{e \in E(\cal{G})})$ is said to be {\it efficient} if the following conditions are satisfied:
\begin{enumerate}
\item $\cal{G}$ is {\it minimal}: if $v = \tau(e)$ is a valence-one vertex, then the edge homomorphism $f_{e}: G_{e}\rightarrow G_{v}$ is not surjective.
\item No {\it invisible vertex}: there is no valence-two vertex $v=\tau(e_{1})=\tau(e_{2})$ $(e_{1} \neq e_{2})$ such that both edge maps $f_{e_{i}}: G_{e_{i}} \rightarrow G_{v}$ $(i=1,2)$ are surjective. 
\item No {\it unused edge}: for every $e \in E(\Gamma)$ the twistor satisfies $z_{e} \neq 1$ (or equivalently $\gamma_{e} \neq \gamma_{\bar{e}}$).
\item No {\it proper power}: if $r^{p} \in f_{e}(G_{e})$ $(p \neq 0)$ then $r \in f_{e}(G_{e})$, for all $e \in E(\Gamma)$.
\item If $v= \tau(e_{1})=\tau(e_{2})$, then $e_{1}$ and $e_{2}$ are not {\em positively bonded}:  for any integers $n_1, n_2 \geq 1$ the elements $f_{e_{1}}(z_{e_{1}}^{n_1})$ and  $f_{e_{2}}(z_{e_{2}}^{n_2})$ are not conjugate in $G_{v}$.
\end{enumerate}
\end{defn}







The following result of Cohen-Lustig  \cite{CL01, CL02} is used crucially in section \ref{sec-6} below.

\begin{thm}[\cite{CL02}]
\label{efficient-uniqueness}
For two efficient Dehn twists $D=D(\cal{G},(z_{e})_{e \in E(\Gamma)})$ and $D^{\prime}=D(\cal{G^{\prime}},(z_{e})_{e \in E(\Gamma^{\prime})})$ one has 
$\hat{D}=\hat{hD'h^{-1}} \in Out(\pi_1(\cal{G}))$ for some isomorphism $h: \pi_1(\cal{G}) \rightarrow \pi_1(\cal{G}')$ if and only if there is a graph-of-groups isomorphism $H: \cal{G} \rightarrow \cal{G}'$ which induces the isomorphism $h$ up to inner automorphism and which takes twistors to twistors, i.e. $H_{e}(z_e)=z_{H_{\Gamma(\cal G)}(e)}$ for all $e \in E(\Gamma(\cal G))$.



\qed
\end{thm}

\subsection{Partial Dehn twists relative to local Dehn twists}
\label{partial-rel-local}

${}^{}$

In \cite{KY02} 
the concept of a
{\em partial Dehn twist relative to a family of local Dehn twists} has been introduced: This is a a partial Dehn twist $D: \cal G \to \cal G$ relative to a subset of vertices $\cal V \subset V(\cal G)$ as defined above, with the additional specification that on any $v \in \cal V$ the graph-of-groups automorphism $D$ is given by a Dehn twist automorphism $D_v: G_v \to G_v$.

The following previous result of the author is crucially used in section~\ref{sec-5} below:

\begin{cor}[Corollary 1.2 of \cite{KY02}]
\label{q-growth+}
Let $\hat \phi \in Out(\FN)$ be represented by a partial Dehn twist 
relative to a family of local Dehn twists.

Then either $\hat \phi$ is itself a Dehn twist automorphism, or else $\hat \phi$ has at least quadratic growth.
\end{cor}

The proof of this corollary is algorithmic, i.e. it can be effectively decided which alternative of the stated dichotomy holds. 
Indeed, a more specific statement is given by the following, 
where we want to stress that it crucially relies on the assumption that the graph $\cal G$ is minimal, see subsection \ref{sec-gog's}:


\begin{cor} [Corollary 7.2 of \cite{KY02}]
\label{rel-D-twists}
Let $\hat \phi \in Out(\FN)$ be represented by a partial Dehn twist $H: \cal G \to \cal G$ relative to a family of local Dehn twists.
Assume that 
%
for some edge $e$ of $\cal G$ the correction term 
$\delta(e)$ 
is not 
locally zero.

Then $\hat \phi$ has at least quadratic growth.
\qed

\end{cor}

\section{Quotes}

In this section we will assemble material from other sources that will be used crucially later.

\subsection{Nielsen-Thurston classification of surface homeomorphisms}
\label{NT-classification}
${}^{}$

The Nielsen-Thurston's classification theorem partitions homotopy classes of homeomorphisms $h$ of a compact 
surface $S$ into three (not mutually exclusive) classes:
(i) {\em periodic}, (ii) {\em reducible} and (iii) {\em pseudo-Anosov}.
It is shown that in each case $h$ can be improved further by an isotopy so that it becomes geometrically very special, with very strong dynamical properties. However, for the purpose of this paper it suffices to note the following consequence of this theory:

\begin{thm}
\label{Thurston+}
Denote $S$ a compact orientable surface.
For any homeomorphism $h: S \rightarrow S$ there is an integer 
$m \geq 1$ 
and a (possibly empty) collection $\cal C$ of pairwise disjoint essential simple closed curves,  
such that $h^{m}: S \rightarrow S$ is isotopic to a homeomorphism $h'$ which preserves a decomposition of $S$ into subsurfaces $S_j$ along $\cal C$, where the restriction $h_j$ of $h'$ to each $S_j$ falls into one of the following three classes:
\begin{enumerate}
\item
$h_j$ is the identity;
\item
$h_j$ is a Dehn twist on an annulus;
\item
$h_j$ is a pseudo-Anosov homeomorphism.
\end{enumerate}
\end{thm}

\begin{rem}
\label{non-orientable}
The above classification result extends in its main parts to a non-orientable surfaces $S$ as well, as can be seen directly from lifting the given homeomorphism to the canonical 2-sheeted orientable covering $\hat S$ of $S$. 
One then applies the above theorem to get $h': \hat S \to \hat S$ and uses the essential uniqueness of $h'$ to argue that $h'$ commutes with the covering involution, so that one can ``quotient'' $h'$  it back to a homeomorphism of $S$.

Here one needs to be a bit careful when it comes to Dehn twists on a curve $\hat c$ in $\hat S$ which is a double cover of a curve $c$ in $S$ with a Moebius band neighborhood $\cal N(c)$. In this case it turns out that a 
twist on $\hat c$ would be isotopic to 
two half-twists on each of the two boundary curves $c_1$ and $c_2$ of $\cal N(\hat c)$ (which is an annulus). However, twisting $n$ times on both, $c_1$ and $c_2$, descends in $\hat S$ to an $n$-fold twist on the boundary curve $c'$ of the Moebius band $\cal N(c)$ in $S$, and it is easy to see that any such twist (including possible a half-twist) is isotopic to the identity in $\cal N(c)$.

As a consequence one can impose for any non-orientable surface the additional condition that the collection $\cal C$ in Theorem \ref{Thurston+} consists only of curves which have an annulus neighborhood.

\end{rem}

If in the situation of Theorem \ref{Thurston+} (or Remark \ref{non-orientable}) none of the $h_j$ falls into class (3), then $h'$ is a {\em multiple Dehn twist} on the collection $\cal C$ of simple closed curves.
It is well-known that Dehn twist homeomorphisms induce linear growth on the conjugacy classes in $\pi_1 S$, while pseudo-Anosov homeomorpisms produce exponential growth. Thus we obtain as direct consequence:

\begin{cor}\label{corquadratic}
(1)
A polynomially growing outer automorphism $\hat \phi$ of $\FN$ which has quadratic or higher growth is not geometric.

\smallskip
\noindent
(2)
Any linearly growing $\hat \phi$ which is geometric 
has a 
positive power that 
is induced by a multiple Dehn twist homeomorphism.
\end{cor}

\subsection{Boundary curves of surfaces}
${}^{}$

Rather than going directly after the question which automorphism of $\FN$ is geometric, one can first consider the following much easier question: 

\smallskip

{\em When is an outer automorphism $\hat \phi$ of $\FN$ with respect to a {\em fixed} identification $\FN = \pi_1 S$ induced by some homeomorphism of $S$ ?}

\smallskip

This classical question has been answered long time ago by combined work of several well known mathematicians (Dehn, Nielsen, Baer, Fenchel, ...):

\begin{thm}[\cite{Zieschang} Theorem~5.7.1 or Theorem~5.7.2]
Let $S$ be a (possibly non-orientable) surface with boundary curves $c_1, \ldots, c_k$,
and assume $k \geq 1$. 

An outer automorphism $\hat \phi$ of $\pi_{1}(S)$ is induced by a homeomorphism of $S$ if and only if there is a permutation $\sigma$ of $\{1, \ldots, k\}$ and exponents $\epsilon_{i} \in \{1, -1\}$ such that 
$\hat \phi$ maps the conjugacy class determined by $c_i$ to the one determined by $c_{\sigma(i)}^{\epsilon_i}$.
\end{thm}

It remains now to study the possible collections of boundary curves in the fundamental group of a surface. For this purpose one has the following well known result:

\begin{thm}\cite{Zieschang}
\label{boundary-collection}
Let $S$ be a compact surface with boundary. Then there exist a basis $\cal B$ for the free group $\pi_1 S$ such that the boundary curves of $S$ 
(up to reversion of their orientation) 
determine homotopy classes in $\pi_1 S$ which are given by the following collection $\cal C$ of elements:
\begin{enumerate}
\item
If $S$ is orientable, then $\cal B = \{s_{1},\ldots, s_{k},u_{1},t_{1},\ldots, u_{g},t_{g}\}$ with $k \geq 1$ and $g \geq 0$, and $\cal C = \{s_{1},\ldots, s_{k}, s_{1}\dots s_{k}\prod\limits_{i=1}^{g}[u_{i},t_{i}]\}$ (where $[x, y] := x y x^{-1} y^{-1}$).
\item
If $S$ is non-orientable, then $\cal B = \{s_{1},\ldots, s_{k},v_{1}, \ldots, v_{\ell}\}$ with $k \geq 1$ and $\ell \geq 0$, and $\cal C = \{s_{1},\ldots, s_{k}, s_{1}\dots s_{k}v_1^2 \ldots v_\ell^2\}$.
\end{enumerate}
\end{thm}

\begin{defn}
\label{bdy-family}
Any family $A = \{a_1, \dots, a_k\}$ of elements $a_i \in \FN$ will be called a {\em boundary family} if there is an automorphism of $\FN$ which maps $A$ to a family of elements that are conjugate to the elements of a 
subset of the 
collection $\cal C$ as given in case (1) or (2) of the above theorem.
\end{defn}

\subsection{Whitehead's Algorithm}
${}^{}$



J.H.C. Whitehead invented in the middle of the last century an algorithm which is one of the strongest, of the most interesting, and also one of most studied among all known algorithms. Although originally deviced for curves on a handlebody, it was quickly understood that its true character is combinatorial; many improved versions of the algorithm have been published since, but they all rely on 
Whitehead's
fundamental insights. A relatively moderate version of it is used in this paper:

\begin{thm}[\cite{Lyndon}]
Given two families of 
elements
$A=\{a_{1},a_{2},...a_{s}\}$ and $B=\{b_{1},b_{2},...,b_{s}\}$ in $F_{n}$, 
it can be decided algorithmically 
whether or not there exists an outer automorphism $\hat\phi \in \Out(F_{n})$ such that 
 $\hat \phi$ maps each conjugacy class $[a_i]$ to $[b_{\sigma(i)}]$, for some permutation $\sigma$.
\end{thm}

Combining the above result with Theorem \ref{boundary-collection} we obtain immediately:

\begin{cor}
\label{bdy-algorithm}
There exists an algorithm which decides whether any given finite family of elements of $\FN$ is a boundary family.
\end{cor}

%




\section{Special topological representatives}
\label{sec-4}

The goal of this section is to derive, for any polynomially growing automorphism $\hat \phi \in \Out(\FN)$, from a relative train track representative of $\hat \phi$ as given by Bestvina-Feighn-Handel, a topological representative with some special properties which are summarized as follows:

\begin{defn}
\label{special-tt}
A 
self map $f: \Gamma \to \Gamma$ of a graph $\Gamma$ which preserves a filtration $V(\Gamma) =\Gamma_0 \subset \Gamma_1 \subset \ldots \subset \Gamma_m = \Gamma$ and induces via some marking isomorphism $\FN \cong \pi_1 \Gamma$ the automorphism $\hat \phi \in \Out(\FN)$ is called a {\em special topological representative of $\hat \phi$} if the following conditions hold:
\begin{enumerate}
\item
Every connected component of $\Gamma_1$ has non-trivial fundamental group and is pointwise fixed by $f$.
\item
Every stratum $H_i = cl(\Gamma_i \smallsetminus \Gamma_{i-1})$ with $i \geq 2$ consists of a single edge $e_i$, with $f(e_i) = w_i e_i u_i$, where $w_i$ and $u_i$ are closed paths in $\Gamma_{i-1}$.
\end{enumerate}
The map $f$ itself will be called a {\em special graph map}.
\end{defn}

Such a special topological representative can be derived algorithmically from 
an improved relative train track
representative of $\hat \phi$ given by Bestvina-Handel-Feighn \cite{BFH00}. This will be explained below in detail; we separate the various issues and treat them in disjoint subsections.

\subsection{Moving all fixed edges into the bottom stratum}
\label{3.1}

${}^{}$

We first recall from Theorem \ref{top-rep} that
%
in \cite{BFH00} it has been shown that for every  polynomially-growing outer automorphism $\hat{\phi} \in Out(F_{n})$ one can derive algorithmically
a topological representative $f: \Gamma \rightarrow \Gamma$ with a filtration $V(\Gamma) =\Gamma_{0} \subset \Gamma_{1} \subset \ldots \subset \Gamma_{m} = \Gamma$ representing a positive power 
of $\hat{\phi}$ with the following properties:
\begin{enumerate}
\item all vertices $v \in V(\Gamma)$ are fixed by the map $f$;
\item every stratum $H_{i} = cl(\Gamma_i \smallsetminus \Gamma_{i-1})$ with $i \geq 1$
consists of a single edge $e_i$ such that $f(e_{i})=e_{i}u_{i}$, where $u_{i} \subset \Gamma_{i-1}$ is a closed 
reduced 
path.
\end{enumerate}

Since we are only interested in the homotopy properties of the filtration and not in the combinatorics, we will now impose the following:

\begin{convention}
\label{u-non-trivial}
\rm
%
We replace the given filtration of $\Gamma$ by a new filtration (denoted homonymously) which has the additional property that all {\em identity edges} $e_i$ (i.e.  $f(e_i) = e_i$) are assembled in the subgraph $\Gamma_1$, referred from now on to as the {\em bottom subgraph}. This is clearly possible simply by relabeling the edge indices, so that we obtain:
\begin{enumerate}
\item[(2*)]
every stratum $H_{i}$ with $i \geq 2$ consists of a single edge $e_i$ such that $f(e_{i})=e_{i}u_{i}$, where $u_{i} \subset \Gamma_{i-1}$ is a closed non-contractible path;
\item[(3)]
the stratum $H_1$ consists entirely of identity edges. 
\end{enumerate}
A connected component of $\Gamma_1$ is called {\em essential} if it has non-trivial fundamental group. Otherwise the connected component is contractible and thus called {\em inessential}.
\end{convention}

\subsection{
The sliding Operation}
\label{3.2}
${}^{}$

In order to improve the filtered topological representative of $\hat \phi$ further we first prove in this subsection a general lemma.

\smallskip

Assume $X$ is a topological space, $P, Q, R \in X$ are points in $X$, and $\gamma \subset X$ is a path which joins $Q$ to $R$.
Then let $Y = X \cup \{e_{1}\}$, $Y^{\prime} = X \cup \{e_{2}\}$ denote two new topological spaces, where $e_{1}$ is an edge which  joins $P$ to $Q$, $e_{2}$ is an edge from $P$ to $R$, and both meet $X$ only in their endpoints.

Suppose $f$ and $f^{\prime}$ are two maps which satisfy the following conditions:

\begin{enumerate}
\item[$\bullet$] $f: Y \rightarrow Y$ is a map such that $f(X) \subset X$ and $f(e_{1})=\beta_{0} e_{1} \beta_{1}$, where $\beta_{0}$ and $\beta_{1}$ are two paths in $X$.

\item[$\bullet$] $f^{\prime}: Y^{\prime} \rightarrow Y^{\prime}$ is a map such that $f^{\prime}|_{X} = f $ and $f^{\prime}(e_{2})=\beta_{0} e_{2} \gamma^{-1} \beta_{1} f(\gamma)$.
\end{enumerate}

\begin{lem}
\label{sliding-lemma}
There is a homotopy equivalence $\kappa: Y \rightarrow Y^{\prime}$ such that the following diagram commutes (up to homotopy):

\[
\begin{CD}
Y @>f>> Y \\
@V\kappa VV @V\kappa VV \\
Y^{\prime} @>>f^{\prime}> Y^{\prime}
\end{CD}
\]

\end{lem}

\begin{proof}
Define $\kappa$ as:
$\kappa |_{X} = id_X$, $\kappa(e_{1})=e_{2} \gamma^{-1}$.
Let now $\kappa^{\prime}: Y^{\prime} \rightarrow Y$ be a map defined through $\kappa^{\prime} |_{X} = id_X$ and $\kappa^{\prime}(e_{2})=e_{1} \gamma$.
Then we can easily verify that $\kappa^{\prime} \circ \kappa 
\simeq
id_{Y}$ and $\kappa \circ \kappa^{\prime} 
\simeq
id_{Y^{\prime}}$. Thus $Y$ and $Y^{\prime}$ are homotopy equivalent.

Now we verify the ``commutative diagram'' $\kappa \circ f 
\simeq
f^{\prime} \circ \kappa$.
The only nontrivial part is to check how they act on $e_{1}$.
%
We observe:
\begin{enumerate}
\item[$\diamond$]
$\kappa \circ f (e_{1}) 
\simeq
\kappa(\beta_{0} e_{1} \beta_{1}) 
\simeq
\beta_{0} e_{2} \gamma^{-1} \beta_{1}$,
\item[$\diamond$]
$f^{\prime} \circ \kappa (e_{1}) \cong f^{\prime} (e_{2} \gamma^{-1}) 
\simeq
\beta_{0} e_{2} \gamma^{-1} \beta_{1} f(\gamma)f(\gamma^{-1}) 
\simeq
\beta_{0} e_{2} \gamma^{-1} \beta_{1}$.
\end{enumerate}
Thus $\kappa \circ f (e_{1})
\simeq
f^{\prime} \circ \kappa (e_{1})$.
Therefore the diagram commutes.
\end{proof}

\begin{rem}
\label{converse-edge}
In our assumption the edges $e_1$ and $e_2$ are oriented from $P$ to $Q$ and from $P$ to $R$ respectively. However, in the next subsection we also need to work with edges directed conversely, in which case the above proof yields the formulas $\kappa(e_{1}^{-1})= \gamma e_{2}^{-1}$, $f(e_{1}^{-1})=\beta_{1}^{-1} e_{1}^{-1} \beta_{0}^{-1}$ and $f'(e_{2}^{-1})=f(\gamma)^{-1} \beta_{1}^{-1} \gamma e_2^{-1} \beta_{0}^{-1}$.

\end{rem}

\subsection{Getting rid of the inessential components}
\label{3.3}
${}^{}$

We go now back to our topological representative of $\hat \phi$ as obtained at the end of subsection \ref{3.1}, i.e. conditions (1), (2*) and (3) are satisfied. The goal of this subsection is to use the sliding lemma from the previous subsection to get rid of all inessential components in the bottom subgraph $\Gamma_1$ of $\Gamma$.

\smallskip

For this purpose we first note that from condition (2*) it follows for every stratum $H_i$ that the sole edge $e_i$ in $H_i$ is attached at its terminal vertex to an essential component of 
$\Gamma_{i-1}$. For the initial vertex, however, this is not clear. Let $H_i$ be the lowest stratum such that the initial vertex of $e_i$ is attached to an inessential component $\cal C$ of 
$\Gamma_{i-1}$.

If no further edge $e_k$ with 
$k \geq i+1$ is attached to $\cal C$, then we can simply contract $\cal C$ together with $e_i$ to the terminal vertex of $e_i$; the left-over deformation retract is $f$-invariant and hence also a topological representative of $\hat \phi$, with all properties as the original one, but less inessential components in the bottom subgraph.

If any edge $e_k$ with $k \geq i+1$ has its initial vertex attached to $\cal C$, 
then we define $\gamma$ to be the path in $\cal C\cup e_i$ from the initial vertex of $e_k$ to the terminal vertex of $e_i$. We then use Lemma \ref{sliding-lemma} to obtain a homotopy equivalent topological representative of $\hat \phi$ through
replacing $e_k$ by an edge $e'_k$ which differs from $e_k$ in that its initial vertex is now equal to the terminal vertex of $e_i$, which lies well in some essential component $\cal C'$. We see from Remark \ref{converse-edge} that $e'_k$ is mapped to a path $w_k {e'}_k u_k$, where $w_k = f(\gamma^{-1}) \gamma = u_i^{-1} \gamma^{-1} \gamma$ which reduces to the path $u_i^{-1}$ that is contained in $\cal C'$.
After this replacement we need to adjust all attaching maps of the strata $H_{k'}$ for $k' \geq k+1$, by composing the given attaching maps of the edge $e_{k'}$ in $H_{k'}$ with the homotopy equivalence $\kappa'$ from Lemma \ref{sliding-lemma}, followed by a homotopy to guarantee that the paths $w_{k'}$ and $u_{k'}$ in the obtained image $w_{k'} e_{k'} u_{k'}$ of the edge $e_{k'}$ do not enter the subgraph $\cal C \cup e_i$.

We now repeat this operation for any edge with initial vertex in $\cal C$, until there is none left over, which allows us to proceed as above in the next to last paragraph. 

After repeating this operation finitely many times no edge will have its initial edge attached to an inessential component of the bottom subgraph. Since the homotopy type of the total graph hasn't changed, and $\Gamma$ was assumed to be connected, we have proved that in the resulting graph the bottom subgraph has no inessential connected component. Thus we have shown:

\begin{prop}
\label{special-prop}
Every polynomially growing automorphism $\hat \phi \in \Out(\FN)$ has a positive power that can be represented by a special topological representative.

\qed
\end{prop}


\section{Iterated Dehn twists}
\label{sec-5}

Let $\hat \phi \in \Out(\FN)$ be any outer automorphism of $\FN$.
The goal of this section is to derive algorithmically 
from a 
special topological 
representative of $\hat \phi$ 
as provided by Proposition \ref{special-prop} 
an {\em iterated Dehn twist} representative of $\hat\phi$.  This is a new object which will be defined in this section.

\medskip

We recall from 
Definition \ref{special-tt}
that a special topological representative 
of $\hat \phi \in \Out(\FN)$ is given through
a graph $\Gamma$ and a filtration $V(\Gamma) =\Gamma_0 \subset \Gamma_1 \subset \ldots \subset \Gamma_m = \Gamma$ 
and a special graph map $f: \Gamma \to \Gamma$
which preserves the filtration
and induces $f$ via some marking isomorphism $\FN \cong \pi_1 \Gamma$ the automorphism $\hat \phi$.
Here ``special graph map'' means that furthermore the following conditions hold:
\begin{enumerate}
\item
Every connected component of $\Gamma_1$ has non-trivial fundamental group and is pointwise fixed by $f$.
\item
Every stratum $\Gamma_j \smallsetminus \Gamma_{i-1}$ with $i \geq 2$ consists of a single edge $e_i$, with $f(e_i) = u_i e_i v_i$, where $u_i$ and $w_i$ are closed paths in $\Gamma_{i-1}$.
\end{enumerate}
%
Notice that from condition (1) 
it
follows immediately that $f$ maps every connected component of any of the
invariant subgraphs
$\Gamma_i$ of $\Gamma$ to itself. Notice also that from condition (2) and the connectedness of $\Gamma$ it follows that $\Gamma_{m-1}$ consists either of a singles connected component $\Gamma_{m-1}^0$ or of two connected components $\Gamma_{m-1}^0$ and $\Gamma_{m-1}^1$. Thus we can always assume that the initial vertex $\iota(e_m)$ of the edge $e_m$ is situated in the component $\Gamma_{m-1}^0$.

\begin{lem}
\label{inductive-step}
Let $f: \Gamma \to \Gamma$ be a special graph map with respect to a filtration  $V(\Gamma) =\Gamma_0 \subset \Gamma_1 \subset \ldots \subset \Gamma_m = \Gamma$. Assume that for every connected component $\Gamma_{m-1}^j$ of $\Gamma_{m-1}$ there is given a group $G_j$, a marking isomorphism $\theta_j: \pi_1(\Gamma_{m-1}^j, v_j) \overset{\cong}{\longrightarrow} G_j$ (for some vertex $v_j \in \Gamma_j$), and 
a group homomorphism $\phi_j: G_j \to G_j$
which satisfies
$\phi_j = \theta_j \circ f_{*, v_j}|_{\pi_1(\Gamma_{m-1}^j, v_j)} \theta_j^{-1}$

Define $\cal G$ to be the graph-of-groups which consists of a single edge $E_m$, as well as a vertex $V_j$ for each connected component $\Gamma_{m-1}^j$ of $\Gamma_{m-1}$, such that the initial or terminal vertex of $E_m$ is attached to $V_j$ if and only if the corresponding vertex of the only edge $e_m$ in the stratum $\Gamma_m \smallsetminus \Gamma_{m-1}$ is attached to the component $\Gamma_{m-1}^j$. 
One defines the edge group $G_{E_m}$ to be trivial, and each vertex groups $G_{V_j}$ to be equal to $G_j$. 

Define a graph-of-groups isomorphism $H_m: \cal G_m \to \cal G_m$ by setting $H_{E_m} = id$ and $H_{V_j} = \phi_j$, and (after recalling the above convention that $\iota(e_m) \in \Gamma_{m-1}^0$) by setting the correction term for the reversed edge $\bar E_m$ equal to 
$\delta({\bar E_m}) :=  
\theta_{0}(\bar\gamma \, u_m \, f(\gamma)) \in G_0$, where 
$\gamma$ denotes a path from ${\tau(\bar e_m)}$ to the vertex $v_{0}$.
The correction term for the edge $E_m$ is given analogously by 
$\delta({E_m}) = 
\theta_j(\bar\gamma' \, w_m \, f(\gamma')) \in G_j$, for $j = 0$ or $j = 1$, 
where we denote by $\gamma'$ a path from ${\tau(e_m)}$ to the vertex $v_j$. 


Then the maps $\psi_0 :=\theta_0^{-1}: G_0 \to \pi_1(\Gamma_{m-1}^0, v_0) \subset \pi_1(\Gamma, v_0)$ and 
$\psi_1: G_j \to  \pi_1(\Gamma, v_0)$ given by $\psi_1(g) := \gamma^{-1} e_m \gamma' \theta_j^{-1}(g) \gamma'^{-1} \bar e_m \gamma$, 
together with the definition $\psi(t_{E_m}) = \gamma^{-1} e_m \gamma' $ in the case $V_1 = V_0$, 
define an isomorphism $\psi:  \pi_1(\cal G, V_0) \to \pi_1(\Gamma, v_0) $ which satisfies 
$$f_{*v_0} = \psi \circ H_{* V_0} \circ \psi^{-1}$$
\end{lem}

\begin{proof}
This is an elementary exercise in chasing through the definition of a graph-of-groups isomorphism (see subsection \ref{sec-gog's}) and all the necessary identifications needed there.

\end{proof}

We note immediately that the graph-of-groups automorphism $H: \cal G \to \cal G$ provided in the last lemma is a partial Dehn twist relative to the full subset $\cal V = V(\cal G)$ of vertices of $\cal G$, see
subsection \ref{sec-gog's}. 
We'd like to point the reader's attention here to the fact that the graph-of-groups $\cal G$ is indeed minimal: The whole point of introducing ``special topological representatives'', and laboring through the previous section in order to get rid of the inessential connected components, is precisely to guarantee this minimality condition.

This gives rise to the following iterative definition:

\begin{defn}
\label{k-level-d-twist}
We first define an 
iterated Dehn twist $D: \cal G \to \cal G$ of level $k = 1$ to be simply a classical graph-of-groups Dehn twist, see 
subsection \ref{sec-gog's}. 
For $k \geq 2$ we define an 
iterated Dehn twist $D: \cal G \to \cal G$ of level $k$ to be a partial Dehn twist relative to $\cal V = V(\cal G)$, where we assume that on each vertex group $G_v$ of $\cal G$ the automorphism $D_v: G_v \to G_v$ is induced by some 
iterated Dehn twist $D^v: \cal G_v \to \cal G_v$ of level $h_v \leq k-1$, through some isomorphism $G_v \cong \pi_1 \cal G_v$.
\end{defn}

By formal reasons we count the identity automorphism as iterated Dehn twist of level $0$.
By definition 
an 
iterated Dehn twist of level $k = 1$ is precisely an traditional Dehn twist as defined in 
subsection \ref{sec-gog's}.
We see immediately that 
an iterated Dehn twist of level $k = 2$ is precisely a 
partial
Dehn twist relative to a family of local Dehn twists as considered in \cite{KY02}, see
subsection \ref{partial-rel-local}. We can now prove:

\begin{prop}
\label{k-level-rep}
Every polynomially growing automorphism $\hat \phi \in \Out(\FN)$ 
has a positive power which
can be represented by an 
iterated Dehn twist $D: \cal G \to \cal G$ of some level $k \geq 0$.

The graph-of-groups $\cal G$ together with all iteratively given data for the vertex groups as well as the automorphism $D$ can be derived algorithmically from any special topological representative $f: \Gamma \to \Gamma$, and hence from $\hat \phi$ itself. If $\Gamma$ has $m$ strata, then one obtains $k \leq m-1$.
\end{prop}

\begin{proof}
The proof of this proposition is done by induction over the number $m \geq 1$ of strata in the given filtration of the special self map $f:\Gamma \to \Gamma$.  

If $m = 1$, then every edge of $\Gamma$ is fixed, so that $\hat \phi$ is the identity automorphism.

We now consider the case $m \geq 2$. As in the situation considered in Lemma \ref{inductive-step}, $\Gamma_{m-1}$ consists of either one or two connected components $\Gamma_{m-1}^j$, and the restriction of $f$ to each of them is a special graph map with $m-1$ or less strata. Thus we can apply our induction hypothesis to obtain 
iterated Dehn twists $D_j: \cal G_j \to \cal G_j$ 
of level $k_j \leq m - 2$
which represent $f|_{\pi_1 \Gamma_{m-1}^j}$.

We now apply Lemma \ref{inductive-step} to derive algorithmically from the edge of the top stratum of $\Gamma$ together with the iterated Dehn twists $D_j$ for the connected components of $\Gamma_{m-1}$ the desired partial Dehn twist representative relative to the family of 
iterated Dehn twists $D_j: \cal G_j \to \cal G_j$ of level $k_j$. This proves our proposition.
%

%

\end{proof}

We now come to the main result of this section:

\begin{thm}
\label{dichotomy0}
Every 
iterated Dehn twist $D: \cal G \to \cal G$ of some level $k \geq 1$ either induces a Dehn twist automorphism on $\pi_1 \cal G$, or else there are conjugacy classes in $\pi_1 \cal G$ which have at least quadratic growth.

This dichotomy can be decided algorithmically.

\end{thm}

\begin{proof}
If $k = 1$, then $D$ is an traditional Dehn twist. For $k \geq 2$ one considers iteratively any of the vertex groups and descends to some sub-iterated Dehn twist of level $k = 2$. One then applies the main result of \cite{KY02} (see Corollary \ref{q-growth+})
which gives precisely the looked-for dichotomy for this sub-iterated Dehn twist. In case where a conjugacy class with at least quadratic growth is found, the proof is finished. In the other case one proceeds with the next sub-iterated  Dehn twist of level $k = 2$. If all of them are induce Dehn twist automorphisms, we can replace all of the given data for the sub-iterated Dehn twists of level 2 by traditional Dehn twists, i.e. iterated Dehn twists of level 1.. This lowers the level of the total iterated Dehn twist from $k$ to $k-1$. Hence proceeding iteratively proves the claim.

\end{proof}

From the last theorem and the previously derived material we obtain:

\begin{cor}
\label{dichotomy0-cor}
Every polynomially growing automorphism $\hat \phi \in \Out(\FN)$ has either polynomial growth of degree $d \geq 2$ or else has a positive power which is a Dehn twist automorphism. This dichotomy can be decided algorithmically.

\qed
\end{cor}

Since every Dehn twist automorphism is known to have linear growth, this also shows Corollary \ref{CLW}
from the Introduction.

\section{Surface homeomorphisms}
\label{sec-6}

Given any collection $\cal C = \{c_1, \ldots, c_r\}$ of pairwise disjoint non-parallel essential curves $c_i$ 
with annulus neighborhood 
in a possibly non-orientable compact surface $S$ with or without boundary, it is well known that $\cal C$ together with the complementary surfaces $S_j \subset S$ define a graph-of-groups $\cal G_{\cal C}$: the underlying graph $\Gamma(\cal G_{\cal C})$
has a vertex $v_j$ for each $S_j$, and for each $c_i$  an edge $e_i$ which can be thought of as ``transverse'' to $c_i$, in the sense that $e_i$ connects the two vertices $v_j$ and $v_{j'}$ which correspond to the two subsurfaces $S_j$ and $S_{j'}$ adjacent to $c_i$. For the vertex groups one sets canonically $G_{v_j} = \pi_1 S_j$, and for the edge groups $G_{e_i} = \pi_1 c_i \cong \Z$, and the edge injections are induced by the topological inclusions of the $c_i$ as boundary curves of the $S_j$. Van Kampen's theorem then gives directly: 
$$\pi_1 \cal G_{\cal C} = \pi_1 S$$

\medskip

For any surface homeomorphism $h: S \to S$ one now uses the groundbreaking Nielsen-Thurston classification for mapping classes (see section \ref{NT-classification}), to obtain a collection $\cal C = \cal C(h)$ of essential simple curves as above, so that (after possibly replacing $h$ by a positive power) every $c_i$ and every $S_j$ is fixed by $h$. More precisely, after improving $h$ by an isotopy, the restriction $h_j: S_j \to S_j$ of $h$ is either the identity homeomorphism, or else it is a pseudo-Anosov automorphism, and furthermore $h$ twists around each $c_i$ an integer number $n_i$ of times. We summarize this in the following well known consequence of Nielsen-Thurston theory:

\begin{prop}
\label{Thurston}
Let $h: S \to S$ be a homeomorphism of a surface $S$ as above. Assume that $S$ has at least one boundary component, and let $\hat\phi \in \Out(\FN)$ be induced by $h$ via some identification isomorphism $\pi_1 S \cong \FN$.

\smallskip
\noindent
(1) If any of the canonical subsurface restrictions $h_j: S_j \to S_j$ of a suitable positive power $h^t$ of $h$ is pesudo-Anosov, then $\hat \phi$ has exponential growth.

\smallskip
\noindent
(2) If none of the $h_j: S_j \to S_j$ is pseudo-Anosov, then 
$h^t$ 
is a multiple surface Dehn twist, and $\hat \phi^{\, t}$ is a Dehn twist automorphism.
\qed
\end{prop}

Indeed, in case (2) of the above proposition, the multiple surface Dehn twist $h^t$ gives immediately rise to a graph-of-groups Dehn twist $D_h: \cal G_{\cal C(h)} \to \cal G_{\cal C(h)}$ on the above graph-of-groups decomposition of $\pi_1 S$ dual to the collection $\cal C(h)$. 
In this case we adopt the convention that any of the curves $c_i \in \cal C$ on which $h^t$ doesn't twist at all is dropped from the collection $\cal C$, so that the twist exponent of every $c_i$ satisfies $n_i \neq 0$.

We now observe:

\begin{lem}
\label{geometric-efficient}
The above Dehn twist $D_h: \cal G_{\cal C(h)} \to \cal G_{\cal C(h)}$ is efficient.
\end{lem}

\begin{proof}
We go through the list of properties of an efficient Dehn twist as stated in Definition \ref{efficient-D-twists-defn}:

\smallskip
\noindent
(1)
$\cal{G}$ is minimal: 
Suppose by contradiction that there exists a valence 1 vertex $v$ with $\tau({e_{i}})=v_j$, and a surjective edge homomorphism $f_{e_{i}}: G_{e_{i}} \rightarrow G_{v_j}$, i.e. we have $G_{v_j} \cong G_{e_{i}} \cong \Bbb{Z}$. This implies that the vertex $v_j$ corresponds to a subsurface $S_j$ which is an annulus.
But then $S_j$ has a second boundary curve, so that $v_j$ would be of valence 2; a contradiction.
%

\smallskip
\noindent
(2)
No invisible vertex:
By construction of the graph-of-groups $\cal{G}$, an invisible vertex only occurs when there exist two simple closed curves $c_{i}$ and $c_{i'}$ are parallel to each other, which is again a contradiction to our assumption.

\smallskip
\noindent
(3)
No unused edges:
This derives from 
our above convention that each twist exponent $n_{i} \neq 0$, since every twistor $z_{e_{i}}$ is non-trivial. 

\smallskip
\noindent
(4)
No proper power:
Given the fact that each $c_{i}$ is an essential simple closed curve in $S$, the induced edge homomorphism $f_{e_i}$ must map the generator $g_{e_i}$ of edge group $G_{e_i} = \<g_{e_i}\>$ to an indivisible element in $G_{v_{j}}$, where $\tau(e_{i})=v_{j}$. 

The fact that the element $f_{e_i}(g_{e_i})$ is indivisible, i.e. it doesn't have a proper root in the group $G_{v_{j}} = \pi_1 S_j$, is a classical fact which can be derived for example from the uniqueness of geodesics in surfaces of constant negative curvature.


\smallskip
\noindent
(5)
Whenever two edges $e_{i}$ and $e_{i'}$ end at the same vertex $v_j$, then $e_{i}$ and $e_{i'}$ are not positively bonded.

Indeed, $e_{i}$ and $e_{i'}$ are neither positively nor negatively bonded, as in either case the corresponding boundary curves in the subsurface $S_j$ corresponding to $v_j$ would have to be parallel, which contradicts our assumptions.
%


\end{proof}

As a consequence of the last lemma we can now test efficiently whether a given Dehn twist automorphism $\hat \phi \in \Out(\FN)$ is geometric or not: It suffices to decide whether or not for each vertex group $G_v$ of some efficient Dehn twist representative $D: \cal G \to \cal G$ of $\hat \phi$ the 
family of 
twistors of the edges adjacent to $v$ define a ``boundary family'' as has been introduced in
Definition \ref{bdy-family}. We recall:


A family $w_1, \ldots, w_r \in \FN$ of elements is called a {boundary family} if there exists a surface $S$ with boundary curves $c_1, \ldots, c_s$ 
with $s \geq r$ 
such that for some identification isomorphism $\theta: \FN \to \pi_1 S $ one has (up to a permutation of the indices of the $c_j$) that every 
$\theta(w_i)$
determines the conjugacy class 
given by $c_i$ in $\pi_1 S$.

\begin{prop}
\label{boundary-set-prop}
An automorphism $\hat \phi \in \Out(\FN)$ represented by an efficient Dehn twist $D: \cal G \to \cal G$ is geometric if and only if for every vertex group $G_v$ of $\cal G$ and for the family of edges $e_i \in E(\cal G)$ with terminal vertex $\tau(e_i) = v$ the corresponding elements 
$f_{e_i}(g_{e_i})$ 
constitute a boundary family in $G_v$.  
Here 
$g_e$
denotes a generator of the cyclic group $G_e$.
\end{prop}

\begin{proof}
If $\hat \phi$ is geometric, then for some surface $S$, some identification isomorphism $\pi_1 S \cong \FN$ and some homeomorphism $h: S \to S$ we have $\hat \phi = h_*$.  We now apply Proposition \ref{Thurston} to $h$. Its alternative (1) is ruled out by our hypothesis that $\hat \phi$ is a Dehn twist automorphism, as those are known to grow linearly (see Theorem \ref{Thurston+} and Corollary \ref{corquadratic}). From alternative (2) we obtain a Dehn twist representative $D_h: \cal G_{\cal C(h)} \to \cal G_{\cal C(h)}$ of some positive power of $h$, which by Lemma \ref{geometric-efficient} is efficient. From the uniqueness for efficient Dehn twist representatives (see Theorem \ref{efficient-uniqueness}) we obtain a graph-of-groups isomorphism $H: \cal G_{\cal C(h)} \to \cal G$ such that $D_h = H^{-1} D^t H$ for some $t \geq 1$. By definition for any vertex group $G_j$ of $\cal G_{\cal C(h)}$ the adjacent edge group generators define a boundary family in $G_j$. Since this property is preserved by the isomorphism $H$, we have shown the ``only if'' part of the claim.

To show the ``if'' direction of the claimed equivalence we use, for every vertex $v$ of $\cal G$, the surface $S_v$ given by the hypothesis on $G_v$ and by Definition \ref{bdy-family}, to construct a surface $S$ by ``tubing together'' the $S_v$ along annuli with core curve $c_i$ as prescribed by the edges $e_i$ of the underlying graph $\Gamma(\cal G)$. From the choice of the generator of each edge group $G_{e_i} \cong \Z$ we deduce the sign $\epsilon_i \in \{1, -1\}$, which together with the twist exponents $n(e_i)$ of $D$ define the homeomorphism $h: S \to S$ as multiple Dehn twist which twists at any $c_i$ precisely $\epsilon_i n(e_i)$ times. It follows directly from this construction that there is a canonical identification isomorphism $\pi_1 S = \FN$ which induces $h_* = \hat \phi$.

\end{proof}

From the last proposition we obtain directly:

\begin{cor}
\label{cor-b-sets}
The geometricity question for Dehn twist automorphisms of $\FN$ can be algorithmically decided if for any finite family $W$ of 
elements 
in a free group $F_m$ it can be algorithmically decided whether $W$ is a boundary family or not.

\qed
\end{cor}

We now obtain as direct consequence of the last result together with Corollary \ref{bdy-algorithm} a direct proof of Corollary \ref{CLW} from the Introduction.




\section{The Algorithm}

In this section we describe concretely the algorithm which decides whether a given automorphism $\hat \phi \in \Out(\FN)$ is of polynomial growth, and whether its power $\hat \phi^{t_n}$ is geometric or not.

\medskip
\noindent
\textbf{Step 1.}
We assume that $\hat \phi$ is given as usually through specifying for some basis $\cal B$ of $\FN$ and some representative $\phi \in \Aut(\FN)$ of $\hat \phi$ the $\phi$-images of the basis elements as words in $\cal B \cup \cal B^{-1}$.

It is shown in \cite{BFH00} how to derive from these data an improved relative train track representative $f: \Gamma \to \Gamma$ of $\hat \phi$, which has furthermore the property that either it contains an exponentially growing stratum, or else its $t_n$-th power has the conditions specified in
Theorem \ref{top-rep}. 
Since these conditions are easy to check 
in finite time, 

at this point we detect whether $\hat \phi$ is of polynomial growth or not. 

In case of a positive answer, we replace for convenience from now on $\hat \phi$ by $\hat \phi^{t_n}$.

\medskip
\noindent
\textbf{Step 2.}
We now transform  $f: \Gamma \to \Gamma$ into a special topological representative of $\hat \phi$ as specified in Definition \ref{special-tt}. For this we follow exactly the procedure explained in section \ref{sec-4}: One first relabels the edges so that every fixed edge now belongs to the bottom stratum. One then moves up from the bottom through all strata, and each time when a stratum $H_i$ consists of an edge $e_i$ which has its initial vertex at some inessential component $\cal C$ of $H_{i-1}$, one performs the sliding operation defined in subsection \ref{3.2} to first replace for any $k \geq i+1$ an edge $e_k$ with initial vertex in $\cal C$ by an edge $e'_k$ that has initial vertex in an essential component. As explained in subsection \ref{3.3}, after each such replacement one has to adjust the attaching maps of any 
edge $e_{k'}$ with $k' \geq k+1$ by a homotopy, 
using the data given concretely by the performed sliding operation. As final operation of this step, after having gone through all strata $H_{k}$ with $k \geq i+1$, we erase the component $\cal C$ together with the edge $e_i$ from the resulting graph.

After finitely many of those operations one has eliminated all inessential components in any subgraph of the given filtration, so that $f: \Gamma \to \Gamma$ is now a special topological representative of $\hat\phi$.

\medskip
\noindent
\textbf{Step 3.}
The next objective would be to derive from the special topological representative $f: \Gamma \to \Gamma$ an iterated Dehn twist representative of $\hat\phi$, as explained in section \ref{sec-5}. This is, however, not the most efficient way, from an algorithmic standpoint. 

Instead we consider only the subgraph $\Gamma'_2$ of $\Gamma$ which is the connected component of $\Gamma_2$ that contains the edge $e_2$. This subgraph is $f$-invariant, and (as shown in the proof of Proposition \ref{k-level-rep}) it defines a partial Dehn twist $D_2: \cal G_2 \to \cal G_2$ relative to $V(\cal G_2)$, where $\cal G_2$ consists of a single edge $E_2$, every vertex group of $\cal G_2$ is given by the fundamental group of the corresponding connected component of $\Gamma'_2 \cap \Gamma_1$ (there are either one or two such components), and the correction terms of $E_2$ and $\bar E_2$ are given by 
$u_2^{-1}$ and $w_2$ respectively, for $f(e_2) = w_2 e_2 u_2$.

Since $f$ acts as identity on $\Gamma_1$, it follows immediately that $D_2$ is an traditional graph-of-groups Dehn twist. At this point we apply the following:

\medskip
\noindent
{\bf Subalgorithm I:}  In \cite{CL02} an algorithm is described that transforms any given Dehn twist into an efficient Dehn twist. We apply this to $D_2$, so that from now on we can assume that $D_2: \cal G_2 \to \cal G_2$ is efficient.

\medskip

We now pass to $\Gamma_3$ and proceed precisely as before for $\Gamma_2$: From $\Gamma'_3$ we construct algorithmically a partial Dehn twist $D_3: \cal G_3 \to \cal G_3$, for which there are two possibilities:  If $\Gamma'_3$ and $\Gamma'_2$ are disjoint, then we are exactly in the same situation as before, so that in this case we obtain an efficient Dehn twist $D_3: \cal G_3 \to \cal G_3$.

If, on the other hand, $\Gamma'_3$ and $\Gamma'_2$ are not disjoint, then $D_3: \cal G_3 \to \cal G_3$ is an iterated Dehn twist of level 2, or in other words, a partial Dehn twist relative to a family of local Dehn twists. In this case we consider the correction terms of $E_3$ and of $\bar E_3$ in the adjacent vertex groups $G_i$, on which $D_3$ acts as (possibly trivial) efficient Dehn twist $D'_i$. We thus can pass to the following:

\medskip
\noindent
{\bf Subalgorithm II (\cite{KY02}) :}  For any efficient Dehn twist $D': \cal G' \to \cal G'$ it can be decided whether (for any vertex $v$ of $\cal G'$) a given element $w \in \pi_1(\cal G', v) \subset \Pi(\cal G)$ is $D'$-conjugate in $\Pi(\cal G')$ to an element of $\cal G'$-length 0. 

\medskip

If one of the correction terms of $D_3$ is not locally zero (i.e. Subalgorithm II gives a negative answer), then it follows from the main result of \cite{KY02} 
(see Corollary \ref{rel-D-twists})
that the automorphism induced by $D_3$ and thus $\hat\phi$ has at least quadratic growth. 
In this case we know that $\hat \phi$ is not geometric (see Corollary \ref{corquadratic}).

(Note that this is the place where we crucially need that there is no inessential component in the bottom subgraph $\Gamma_1$, as otherwise it could happen that $\cal G_3$ is not minimal, and in this case Corollary \ref{rel-D-twists} would fail to hold.)

If, on the other hand, all correction terms are locally zero, we can pass to the following:

\medskip
\noindent
{\bf Subalgorithm III (\cite{KY01}):}  Every partial Dehn twist $D: \cal G \to \cal G$ relative to a family of local Dehn twists, for which all correction terms are locally zero, induces on $\pi_1\cal G$ an traditional Dehn twist automorphism. A graph-of-groups Dehn twist representative $D'$ of the outer automorphism $\hat D$ can be derived algorithmically from $D$.

\medskip

Thus we can transform $D_3$ algorithmically first into an traditional Dehn twist, and then apply Subalgorithm I to make it efficient. 

We then pass to $\Gamma_4$, and repeat the above procedure iteratively, going through all strata of $\Gamma$.  As a result we either obtain that $\hat \phi$ has at least quadratic growth and hence is not geometric, or else we have derived an efficient Dehn twist representative $D: \cal G \to \cal G$ for $\hat \phi$.

\medskip
\noindent
\textbf{Step 4.}
We now turn to section \ref{sec-6}: It has been shown in Proposition \ref{boundary-set-prop} that $\hat \phi$ is geometric if and only if for every vertex group $G_v$ of $\cal G$ the family of edge group generators for the edges adjacent to the vertex $v$ define a boundary family in $G_v$. This is a question that can be decided algorithmically through the Whitehead algorithm, see Corollary \ref{cor-b-sets}. This finishes the algorithm.


\begin{thebibliography}{ABC}

\bibitem{Bass} Hyman Bass, \emph{Covering theory for graphs of groups.}
J.Pure Appl.Algebra 89, 3-47, 1993.

\bibitem{BH92} M.~Bestvina, and M.~Handel, \emph{Train tracks and
automorphisms of free groups.} Ann.  of Math.  (2) \textbf{135}
(1992), no.  1, 1--51

\bibitem{BFH00}
M.~Bestvina, M.~Feighn, and M.~Handel, \emph{The Tits alternative for
${\rm Out}(F\sb n)$.  I. Dynamics of exponentially-growing
automorphisms.} Ann.  of Math.  (2) \textbf{151} (2000), no.  2,
517--623

\bibitem{BFH05}
M.~Bestvina, M.~Feighn, and M.~Handel, \emph{The Tits alternative for
${\rm Out}(F\sb n)$.  II. A Kolchin type theorem.} Ann.  of Math.  (2)
\textbf{161} (2005), no.  1, 1--59

\bibitem{Bo} Oleg~Bogopolski, \emph{Introduction to Group Theory.}

\bibitem{CL01} Marshall M.~Cohen and Martin~Lustig, \emph{Very small group actions on
$R$-trees and Dehn twist automorphisms.} 
Topology 34(3), 575--617, 1995.

\bibitem{CL02} Marshall M.~Cohen and Martin~Lustig, \emph{The conjugacy problem for Dehn twist automorphisms of free groups.} 
Comment. Math. Helv. 74(2), 179 - 200, 1999.

\bibitem{FM} Benson~Farb and Dan~Magalit, \emph{A primer on mapping class groups.}

http://www.maths.ed.ac.uk/~aar/papers/farbmarg.pdf

\bibitem{Lyndon} Roger C. Lyndon and Paul E. Schupp, \emph{Combinatorial Group Theory.}
Springer-Verlag, Berlin-New York, 1997.

\bibitem{RW} Moritz~Rodenhausen and Richard D.~Wade, \emph{Centraliser of Dehn twist automorphisms of free groups.} 

http://arxiv.org/pdf/1206.5728.pdf

\bibitem{Rodenhausen} Moritz Rodenhausen, \emph{Centralisers of polynomially growing automorphisms of free groups.} 
PhD thesis, 2013.

http://www2.warwick.ac.uk/fac/sci/maths/people/staff/rodenhausen/phdthesis.pdf

\bibitem{Serre} Jean-Pierre~Serre, \emph{Trees.}
Springer, 1980.

\bibitem{Swarup} Gadde A. Swarup, \emph{Decomposition of free groups.} 
J.Pure Appl.Algebra 40, 99-102, 1986.

\bibitem{KY01} Kaidi~Ye, {\em Quotient and Blow-up of Automorphisms of Graph-of-groups.} preprint 2015.

\bibitem{KY02} Kaidi~Ye, {Partial Dehn twists of free groups relative to local Dehn twists - a dichotomy.} preprint 2016.

\bibitem{Zieschang} Heiner Zieschang, \emph{Surfaces and Planar Discontinuous Groups.} 

\end{thebibliography}
\end{document}